\documentclass[12pt]{amsart}
\usepackage{amscd}
\usepackage{amsmath}
\usepackage{verbatim}

\usepackage[cmtip,arrow]{xy}

\usepackage{pb-diagram,pb-xy}

\input cyracc.def \font\tencyr=wncyr10 \def\cyr{\tencyr\cyracc}

\usepackage{amssymb}

\numberwithin{equation}{section}

\newtheorem{thm}[equation]{Theorem}
\newtheorem{prop}[equation]{Proposition}

\newtheorem{lemma}[equation]{Lemma}
\newtheorem{cor}[equation]{Corollary}

\theoremstyle{definition}
\newtheorem{rem}[equation]{Remark}

\newtheorem{dfn}[equation]{Definition}

\newcommand{\Steen}{\operatorname{Sq}}

\newcommand{\rdim}{\operatorname{rdim}}

\newcommand{\SB}{X}

\newcommand{\Br}{\mathop{\mathrm{Br}}}

\newcommand{\ind}{\mathop{\mathrm{ind}}}
\newcommand{\coind}{\mathop{\mathrm{coind}}}

\newcommand{\CH}{\mathop{\mathrm{CH}}\nolimits}

\newcommand{\BCH}{\mathop{\overline{\mathrm{CH}}}\nolimits}

\newcommand{\Ch}{\mathop{\mathrm{Ch}}\nolimits}

\newcommand{\res}{\mathop{\mathrm{res}}\nolimits}

\newcommand{\pr}{\operatorname{\mathit{pr}}}

\newcommand{\mult}{\operatorname{mult}}

\newcommand{\id}{\mathrm{id}}

\newcommand{\Z}{\mathbb{Z}}
\newcommand{\F}{\mathbb{F}}

\newcommand{\Spec}{\operatorname{Spec}}
\newcommand{\End}{\operatorname{End}}
\newcommand{\Hom}{\operatorname{Hom}}

\newcommand{\Oplus}{\operatornamewithlimits{\textstyle\bigoplus}}

\renewcommand{\>}{\right>}

\newcommand{\compose}{\circ}

\newcommand{\CM}{\operatorname{CM}}
\newcommand{\BCM}{\operatorname{\overline{CM}}}

\renewcommand{\phi}{\varphi}

\newcommand{\RatM}{\dashrightarrow}

\usepackage[hypertex]{hyperref}

\title
{Hyperbolicity of orthogonal involutions}

\keywords
{Algebraic groups,
projective homogeneous varieties,
Chow groups.\\
{\em 2000 Mathematical Subject Classifications:}
14L17; 14C25}

\author
{Nikita A. Karpenko}

\address
{UPMC Univ Paris 06\\
Institut de Math\'ematiques de Jussieu\\
F-75252 Paris\\
FRANCE}

\address
{{\it Web page:}
{\tt www.math.jussieu.fr/\~{ }karpenko}}

\email {karpenko {\it at} math.jussieu.fr}

\date{18 April  2009}

\thanks
{Supported by the Collaborative Research Centre 701 of the Bielefeld University}

\begin{document}

\begin{abstract}
We show that a non-hyperbolic
orthogonal involution on a central simple algebra
over a field of characteristic $\ne2$ remains
non-hyperbolic over
some splitting field of the algebra.
\end{abstract}

\maketitle


\section
{Introduction}

Throughout this note (besides of
\S\ref{Krull-Schmidt principle} and \S\ref{Splitting off a motivic summand})
$F$ is a field of characteristic $\ne2$.
The basic reference for the material related to involutions on central simple algebras
is \cite{MR1632779}.
The {\em degree} $\deg A$ of a (finite-dimensional) central simple $F$-algebra
$A$ is the integer $\sqrt{\dim_FA}$;
the {\em index} $\ind A$ of $A$ is the degree of a central division algebra Brauer-equivalent to $A$.

The main result of this paper is as follows (the proof is given in \S\ref{Proof of Theorem}):

\begin{thm}[Main theorem]
\label{main}
A non-hyperbolic orthogonal involution $\sigma$ on a central simple $F$-algebra $A$
remains non-hyperbolic over the function field of the Severi-Brauer variety
of $A$.
\end{thm}

To explain the statement of Abstract, let us
note that the function field $L$ of the Severi-Brauer variety of a central simple algebra $A$ is a {\em splitting field} of $A$,
that is, the $L$-algebra $A_L$ is Brauer-trivial.

A stronger version of Theorem \ref{main}, where the
word ``non-hyperbolic'' (in each of two appearances) is replaced by ``anisotropic'', is, in general, an open conjecture, cf.
\cite[\Small Conjecture 5.2]{MR1751923}.

Let us recall that the index of a central simple algebra possessing an orthogonal involution is a power of $2$.
Here is the complete list of indices $\ind A$ and coindices
$\coind A=\deg A/\ind A$ of $A$
for which Theorem \ref{main} is known (over arbitrary fields
of characteristic $\ne2$), given in the chronological order:
\begin{itemize}
\item
$\ind A=1$ --- trivial;
\item
$\coind A=1$ (the stronger version) --- \cite[\Small Theorem 5.3]{MR1751923};
\item
$\ind A=2$ (the stronger version) --- \cite[\Small Corollary 3.4]{MR1850658};
\item
$\coind A$ odd ---  \cite[\Small appendix by Zainoulline]{kirill} and independently
\cite[\Small Theorem 3.3]{isotropy};
\item
$\ind A=4$ and $\coind A=2$ --- \cite[\Small Proposition 3]{MR2128421};
\item
$\ind A=4$ --- \cite[\Small Theorem 1.2]{hyper}.
\end{itemize}

Let us note that
Theorem \ref{main} for any given $(A,\sigma)$ with $\coind A=2$
implies the stronger version version  of Theorem \ref{main} for this
$(A,\sigma)$:
indeed, by \cite[\Small Theorem 3.3]{isotropy},
if $\coind A=2$ and $\sigma$ becomes isotropic over the function field of the
Severi-Brauer variety, then $\sigma$ becomes hyperbolic over this function
field and the weaker version applies.
Therefore we get

\begin{thm}
An anisotropic orthogonal
involution on a central simple $F$-algebra of coindex $2$
remains anisotropic over the function field of the Severi-Brauer variety of the algebra.
\qed
\end{thm}

Sivatski's proof of the case with $\deg A=8$ and $\ind A =4$, mentioned above,
is based on the following theorem, due to Laghribi:

\begin{thm}[{\cite[\Small Th{\'e}or{\`e}me 4]{MR1417621}}]
\label{laghribi}
Let $\phi$ be an anisotropic quadratic form of dimension $8$ and of
trivial discriminant.
Assume that the index of the Clifford algebra $C$ of $\phi$ is $4$.
Then $\phi$ remains anisotropic over the function field $F(X_1)$
of the Severi-Brauer variety
$X_1$ of $C$.
\end{thm}

The following alternate proof of Theorem \ref{laghribi}, given by
Vishik, is a prototype of our proof of Main theorem (Theorem \ref{main}).
Let $Y$ be the projective quadric of $\phi$ and let $X_2$ be the
Albert quadric of a biquaternion division algebra Brauer-equivalent to $C$.
Assume that $\phi_{F(X_1)}$ is isotropic.
Then
for any field extension $E/F$,
the Witt index of $\phi_E$ is at least $2$ if and only if
$X_2(E)\ne\emptyset$.
By \cite[\Small Theorem 4.15]{MR2066515} and since the Chow motive $M(X_2)$ of $X_2$ is
indecomposable,
it follows that the motive $M(X_2)(1)$ is a summand of
the motive of $Y$.
The complement summand of $M(Y)$ is then
given by a {\em Rost projector} on $Y$ in the sense of Definition \ref{def rost corr}.
Since $\dim Y+1$ is not a power of $2$,
it follows that $Y$ is isotropic (cf. \cite[\Small Corollary 80.11]{EKM}).

After introducing some notation in \S\ref{Notation} and discussing some important general principles
concerning Chow motives in \S\ref{Krull-Schmidt principle},
we produce in \S\ref{Splitting off a motivic summand}  a replacement
of \cite[\Small Theorem 4.15]{MR2066515} (used right above
to split off the summand $M(X_2)(1)$ from the motive of $Y$) valid for more general (as projective quadrics) algebraic varieties.
In \S\ref{Rost correspondences} we reproduce some recent results due to Rost concerning the modulo $2$ Rost correspondences and Rost projectors on more general (as projective quadrics) varieties.
In \S\ref{Motivic decompositions of some isotropic varieties}
we apply some standard motivic decompositions of projective
homogeneous varieties to certain varieties related to a central simple algebra
with an isotropic orthogonal involution.
We also reproduce (see Theorem \ref{cd(gSB)}) some results of \cite{outmot}
which contain the needed generalization of indecomposability of the motive of an Albert quadric used in the previous paragraph.
Finally, in \S\ref{Proof of Theorem} we prove Main theorem (Theorem \ref{main}) following the strategy of \cite{hyper}
and using results of \cite{outmot} which were not available at the time of \cite{hyper}.

\bigskip
\noindent
{\sc Acknowledgements.}
Thanks to Anne Qu{\'e}guiner for asking me the question and to Alexander Vishik for telling me the alternate proof of Theorem
\ref{laghribi}.

\section
{Notation}
\label
{Notation}

We understand under a {\em variety} a separated scheme of finite type over
a field.

Let $D$ be a central simple $F$-algebra.
The $F$-dimension of any right ideal in $D$ is divisible by $\deg D$;
the quotient is the {\em reduced dimension} of the ideal.
For any integer $i$, we write $\SB(i;D)$
for the generalized Severi-Brauer variety of the right ideals in $D$ of reduced
dimension $i$.
In particular, $\SB(0;D)=\Spec F=\SB(\deg D;D)$ and
$\SB(i,D)=\emptyset$ for $i<0$ and for $i>\deg D$.

More generally, let $V$ be a right $D$-module.
The $F$-dimension of $V$ is then divisible by $\deg D$ and the quotient
$\rdim V=\dim_F V/\deg D$ is
called the {\em reduced dimension} of $V$.
For any integer $i$, we write $\SB(i;V)$
for the projective homogeneous variety of the $D$-submodules in $V$ of reduced
dimension $i$ (non-empty iff $0\leq i\leq\rdim V$).
For a finite sequence of integers $i_1,\dots,i_r$, we write
$\SB(i_1\subset\dots\subset i_r;V)$
for the projective homogeneous variety of flags of the $D$-submodules in $V$ of reduced
dimensions $i_1,\dots,i_r$ (non-empty iff $0\leq i_1\leq\dots\leq i_r\leq\rdim V$).

Now we additionally assume that $D$ is endowed with an orthogonal involution
$\tau$.
Then we write $\SB(i;(D,\tau))$
for the variety of the totally isotropic right ideals in $D$ of reduced
dimension $i$ (non-empty iff $0\leq i\leq \deg D/2$).

If moreover $V$ is endowed with a hermitian (with respect to $\tau$) form $h$,
we write $\SB(i;(V,h))$
for the variety of the totally isotropic $D$-submodules in $V$ of reduced
dimension $i$.

We refer to \cite{MR1758562} for a detailed construction and basic properties of the above varieties.


\section
{Krull-Schmidt principle}
\label
{Krull-Schmidt principle}

The characteristic of the base field $F$ is arbitrary in this section.

Our basic reference for Chow groups and Chow motives (including notation)
is \cite{EKM}.
We fix an
associative unital commutative ring $\Lambda$
(we shall take $\Lambda=\F_2$ in the application)
and for a variety $X$ we write $\CH(X;\Lambda)$ for its Chow group with coefficients in
$\Lambda$.
Our category of motives is
the category $\CM(F,\Lambda)$ of {\em graded Chow motives with
coefficients in $\Lambda$,} \cite[\Small definition of \S64]{EKM}.
By a {\em sum} of motives we always mean the {\em direct} sum.

We shall often assume that our coefficient ring $\Lambda$ is finite.
This simplifies significantly the situation (and is sufficient for our application).
For instance, for a finite $\Lambda$, the endomorphism rings of finite sums
of Tate motives are also finite and the following easy statement applies:

\begin{lemma}
\label{finite}
An appropriate power of any element of any {\em finite} associative (not necessarily
commutative) ring is idempotent.
\end{lemma}

\begin{proof}
Since the ring is finite, any its element $x$ satisfies $x^a=x^{a+b}$
for some $a\geq1$ and $b\geq1$.
It follows that $x^{ab}$ is an idempotent.
\end{proof}

Let $X$ be a smooth complete variety over $F$.
We call $X$ {\em split}, if its {\em integral} motive $M(X)\in\CM(F,\Z)$ (and therefore its motive with any coefficients)
is a finite sum of Tate motives.
We call $X$ {\em geometrically split}, if it splits over
a field extension of $F$.
We say that $X$ satisfies the {\em nilpotence principle}, if for any field extension
$E/F$ and any coefficient ring $\Lambda$,
 the kernel of the change of field homomorphism $\End(M(X))\to\End(M(X)_E)$ consists
of nilpotents.
Any projective homogeneous variety is geometrically split and satisfies the nilpotence
principle, \cite[\Small Theorem 8.2]{MR2110630}.

\begin{cor}[{\cite[\Small Corollary 2.2]{outmot}}]
\label{cor-finite}
Assume that the coefficient ring $\Lambda$ is finite.
Let $X$ be a geometrically split variety satisfying the nilpotence principle.
Then an appropriate power of any endomorphism of the motive of $X$
is a projector.
\end{cor}

We say that the {\em Krull-Schmidt principle} holds for a given pseudo-abelian category,
if every object of the category has one and unique decomposition in a finite direct
sum of indecomposable objects.
In the sequel, we are constantly using the following statement:

\begin{cor}[{\cite[\Small Corollary 35]{MR2264459}, see also \cite[\Small Corollary 2.6]{outmot}}]
\label{krull-schmidt}
Assume that the coefficient ring $\Lambda$ is finite.
The Krull-Schmidt principle holds for the pseudo-abelian Tate subcategory in
$\CM(F,\Lambda)$
generated by the motives of the geometrically split $F$-varieties satisfying the
nilpotence principle.
\qed
\end{cor}

\begin{rem}
Replacing the Chow
groups $\CH(-;\Lambda)$ by the {\em reduced} Chow groups
$\BCH(-;\Lambda)$ (cf. \cite[\Small \S72]{EKM})
in the definition of the category $\CM(F,\Lambda)$,
we get a ``simplified'' motivic category $\BCM(F,\Lambda)$
(which is still sufficient for the main purpose of this paper).
Working within this category, we do not need the nilpotence principle any more.
In particular, the Krull-Schmidt principle holds (with a simpler proof)
for the pseudo-abelian Tate subcategory in
$\BCM(F,\Lambda)$
generated by the motives of the geometrically split $F$-varieties.
\end{rem}

\section
{Splitting off a motivic summand}
\label
{Splitting off a motivic summand}

The characteristic of the base field $F$ is still arbitrary in this section.

In this section we assume that the coefficient ring $\Lambda$ is connected.
We shall often assume that $\Lambda$ is finite.

Before climbing to the main result of this section (which is Proposition \ref{prop}), let us do some warm up.

The following definition of \cite{outmot} extends some terminology of \cite{Vishik-IMQ}:

\begin{dfn}
\label{def-outer}
%
Let $M\in\CM(F,\Lambda)$ be a summand of the motive of a smooth complete
irreducible variety of dimension $d$.
The summand $M$ is called {\em upper}, if $\CH^0(M;\Lambda)\ne0$.
The summand $M$ is called {\em lower}, if $\CH_d(M;\Lambda)\ne0$.
The summand $M$ is called {\em outer}, if it is simultaneously upper and lower.
\end{dfn}

For instance, the whole motive of a smooth complete irreducible
variety is an outer summand of itself.
Another example of an outer summand is the motive given by a {\em Rost projector} (see
Definition \ref{def rost corr}).

Given a correspondence $\alpha\in\CH_{\dim X}(X\times Y;\Lambda)$ between some smooth complete irreducible varieties
$X$ and $Y$, we write $\mult\alpha\in\Lambda$ for the {\em multiplicity} of $\alpha$,
\cite[\Small definition of \S75]{EKM}.
Multiplicity of a composition of two correspondences is the product of multiplicities of the composed
correspondences (cf. \cite[\Small Corollary 1.7]{MR1751923}).
In particular, multiplicity of a projector is idempotent and therefore $\in\{0,1\}$ because the coefficient
ring $\Lambda$ is connected.

Characterizations of outer summands given in the two following Lemmas are easily obtained:

\begin{lemma}[{cf. \cite[\Small Lemmas 2.8 and 2.9]{outmot}}]
\label{p outer}
Let $X$ be a smooth complete irreducible variety.
The motive $(X,p)$ given by a projector $p\in\CH_{\dim X}(X\times X;\Lambda)$
is upper if and only if $\mult p=1$.
The motive $(X,p)$
is lower if and only if $\mult p^t=1$, where $p^t$ is the transpose of $p$.
\end{lemma}


\begin{lemma}[{cf. \cite[\Small Lemma 2.12]{outmot}}]
\label{tate outer}
Assume that a summand $M$ of the motive of a smooth complete irreducible
variety of dimension $d$ decomposes into a sum of Tate motives.
Then $M$ is upper if and only if the Tate motive $\Lambda$ is present in the
decomposition;
it is lower if and only if the Tate motive $\Lambda(d)$ is present in the decomposition.
\end{lemma}


The following statement generalizes
(the finite coefficient version of)
\cite[\Small Corollary 3.9]{MR2066515}:

\begin{lemma}
\label{tuda-suda}
Assume that the coefficient ring $\Lambda$ is finite.
Let $X$ and $Y$ be smooth complete irreducible varieties such that
there exist multiplicity $1$ correspondences
$$
\alpha\in\CH_{\dim X}(X\times Y;\Lambda)\;\;
\text{ and } \;\;\beta\in\CH_{\dim Y}(Y\times X;\Lambda).
$$
Assume that $X$ is geometrically split and satisfies the nilpotence principle.
Then there is an upper summand of $M(X)$ isomorphic to an upper summand of $M(Y)$.
Moreover, for any upper summand $M_X$ of $M(X)$ and any upper summand $M_Y$ of $M(Y)$,
there is an upper summand of $M_X$ isomorphic to an upper summand of $M_Y$.
\end{lemma}

\begin{proof}
By Corollary \ref{cor-finite}, the composition $p:=(\beta\compose\alpha)^{\compose n}$ for some $n\geq1$
is a projector.
Therefore $q:=(\alpha\compose\beta)^{\compose 2n}$ is also a projector and
the summand $(X,p)$ of $M(X)$ is isomorphic to the summand $(Y,q)$ of $M(Y)$:
mutually inverse isomorphisms are, say,
$$
\alpha\compose(\beta\compose\alpha)^{\compose(2n)}:(X,p)\to(Y,q)\;\text{ and }\;
\beta\compose(\alpha\compose\beta)^{\compose(4n-1)}:(Y,q)\to(X,p)\;.
$$
Since $\mult p=(\mult\beta\cdot\mult\alpha)^n=1$ and similarly
$\mult q=1$, the summand $(X,p)$ of $M(X)$ and the summand $(Y,q)$ of $M(Y)$
are upper by Lemma \ref{p outer}.

We have proved the first statement of Lemma \ref{tuda-suda}.
As to the second statement,
let
$$
p'\in\CH_{\dim X}(X\times X;\Lambda)\;\text{ and }\;q'\in\CH_{\dim Y}(Y\times Y;\Lambda)
$$
be projectors such that $M_X=(X,p')$ and $M_Y=(Y,q')$.
Replacing $\alpha$ and $\beta$ by $q'\compose\alpha\compose p'$ and $p'\compose\beta\compose q'$,
we get isomorphic upper motives $(X,p)$ and $(Y,q)$ which are summands of $M_X$ and $M_Y$.
\end{proof}

\begin{rem}
\label{outer unique}
Assume that the coefficient ring $\Lambda$ is finite.
Let $X$ be a geometrically split irreducible smooth complete
variety satisfying the nilpotence principle.
Then the complete motivic decomposition of $X$ contains precisely one upper
summand and it follows by Corollary \ref{krull-schmidt} (or by Lemma \ref{tuda-suda})
that an upper indecomposable summands of $M(X)$ is unique up to
an isomorphism.
(Of course, the same is true for the lower summands.)
\end{rem}

Here comes the needed replacement of \cite[\Small Theorem 4.15]{MR2066515}:

\begin{prop}
\label{prop}
Assume that the coefficient ring $\Lambda$ is finite.
Let $X$ be a geometrically split,
geometrically irreducible variety satisfying the nilpotence principle and let
$M$ be a motive.
Assume that there exists a field extension $E/F$ such that
\begin{enumerate}
\item
the field extension $E(X)/F(X)$ is purely transcendental;
\item
the upper indecomposable summand of $M(X)_E$ is also lower and is a summand of $M_E$.
\end{enumerate}
Then the upper indecomposable summand of $M(X)$ is a summand of $M$.
\end{prop}

\begin{proof}
We may assume that $M=(Y,p,n)$ for some irreducible smooth complete $F$-variety $Y$, a
projector $p\in\CH_{\dim Y}(Y\times Y;\Lambda)$, and an integer $n$.

By the assumption (2), we have
morphisms of motives $f:M(X)_E\to M_E$ and $g:M_E\to M(X)_E$
with
$\mult(g\compose f)=1$.
By \cite[\Small Lemma 2.14]{outmot},
in order to prove Proposition \ref{prop}, it suffices to
construct morphisms $f':M(X)\to M$ and $g':M\to M(X)$ (over
$F$) with $\mult(g'\compose f')=1$.

Let $\xi:\Spec F(X)\to X$ be the generic point of the (irreducible) variety $X$.
For any $F$-scheme $Z$, we write $\xi_Z$ for the
morphism $\xi_Z=(\xi\times\id_Z):Z_{F(X)}=\Spec_{F(X)}\times Z\to X\times Z$.
Note that for any $\alpha\in\CH(X\times Z)$, the image
$\xi_Z^*(\alpha)\in\CH(Z_{F(X)})$ of $\alpha$ under the pull-back homomorphism
$\xi_Z^*:\CH(X\times Z,\Lambda)\to\CH(Z_{F(X)},\Lambda)$
coincides with the composition of correspondences $\alpha\compose[\xi]$,
\cite[\Small Proposition 62.4(2)]{EKM},
where $[\xi]\in\CH_0(X_{F(X)},\Lambda)$ is the class of the point $\xi$:
\begin{equation*}
\tag{$*$}
\xi_Z^*(\alpha)=\alpha\compose[\xi].
\end{equation*}

In the commutative square
$$
\begin{CD}
\CH(X_E\times Y_E;\Lambda)@>{\xi^*_{Y_E}}>>\CH(Y_{E(X)};\Lambda)\\
@A{\res_{E/F}}AA @A{\res_{E(X)/F(X)}}AA\\
\CH(X\times Y;\Lambda)@>{\xi^*_Y}>>\CH(Y_{F(X)};\Lambda)
\end{CD}
$$
the change of field homomorphism $\res_{E(X)/F(X)}$ is surjective\footnote{In
fact, $\res_{E(X)/F(X)}$ is even an isomorphism, but we do not need its injectivity
(which can be obtained with a help of a specialization).}
 because
of the assumption (1) by the homotopy invariance of Chow groups
\cite[\Small Theorem 57.13]{EKM} and
by the localization property of Chow groups
\cite[\Small Proposition 57.11]{EKM}.
Moreover, the pull-back homomorphism $\xi^*_Y$ is
surjective by \cite[\Small Proposition 57.11]{EKM}.
It follows that there exists an element $f'\in\CH(X\times Y;\Lambda)$ such
that $\xi^*_{Y_E}(f'_E)=\xi^*_{Y_E}(f)$.

Recall that $\mult(g\compose f)=1$.
On the other hand,
$\mult(g\compose f'_E)=\mult(g\compose f)$.
Indeed,
$\mult(g\compose f)=\deg\xi_{X_E}^*(g\compose f)$
by \cite[\Small Lemma 75.1]{EKM}, where $\deg:\CH(X_{E(X)})\to\Lambda$ is the degree homomorphism.
Furthermore, $\xi_{X_E}^*(g\compose f)=(g\compose f)\compose[\xi_E]$ by ($*$).
Finally, $(g\compose f)\compose[\xi_E]=g\compose(f\compose[\xi_E])$
and $f\compose[\xi_E]=\xi_{Y_E}^*(f)=\xi_{Y_E}^*(f'_E)$ by the construction of $f'$.


Since $\mult(g\compose f'_E)=1$ and the indecomposable upper
summand of $M(X)_E$ is lower, we have
$\mult((f'_E)^t\compose g^t)=1$.
Therefore we may apply the above procedure to the dual morphisms
\begin{multline*}
g^t:M(X)_E\to (Y,p,n+\dim X-\dim Y)_E
\\
\text{ and } \;\;\;\;
(f'_E)^t:(Y,p,n+\dim X-\dim Y)_E\to M(X)_E.
\end{multline*}
This way we get
a morphism $g':M\to M(X)$ such that
$\mult((f')^t\compose(g')^t)=1$.
It follows that $\mult(g'\compose f')=1$.
\end{proof}


\begin{rem}
\label{BCM}
Replacing $\CM(F,\Lambda)$ by $\BCM(F,\Lambda)$ in Proposition \ref{prop},
we get a weaker version of Proposition \ref{prop} which is still
sufficient for our application.
The nilpotence principle is no more needed in the proof of the weaker
version.
Because of that, there is no more need to assume that $X$ satisfies the nilpotence
principle.
\end{rem}

\section
{Rost correspondences}
\label
{Rost correspondences}

In this section, $X$ stands for a smooth complete geometrically irreducible variety of a positive dimension $d$.

The coefficient ring $\Lambda$ of the motivic category is $\F_2$ in this section.
We write $\Ch(-)$ for the Chow group $\CH(-;\F_2)$ with coefficients in $\F_2$.
We write $\deg_{X/F}$ for the degree homomorphism $\Ch_0(X)\to\F_2$.


\begin{dfn}
\label{def rost corr}
An element $\rho\in\Ch_d(X\times X)$ is called a {\em Rost
correspondence} (on $X$), if $\rho_{F(X)}=\chi_1\times[X_{F(X)}]+[X_{F(X)}]\times\chi_2$
for
some $0$-cycle classes
$\chi_1,\chi_2\in\Ch_0(X_{F(X)})$ of degree $1$.
A {\em Rost projector} is a Rost correspondence which is a projector.
\end{dfn}

\begin{rem}
Our definition of a Rost correspondence differs from the definition of a {\em special
correspondence} in \cite{markus}.
Our definition is weaker in the sense that a special correspondence on $X$ (which is an element of the {\em integral} Chow group $\CH_d(X\times X)$) considered modulo $2$ is a Rost correspondence but not any Rost correspondence is obtained this way.
This difference gives a reason to reproduce below some results of \cite{markus}.
Actually, some of the results below are formally more general than the corresponding results of \cite{markus};
their proofs, however, are essentially the same.
\end{rem}

\begin{rem}
Clearly, the set of all Rost correspondences on $X$ is stable under transposition and composition.
In particular, if $\rho$ is a Rost correspondence, then its both symmetrizations
$\rho^t\compose\rho$ and $\rho\compose\rho^t$ are (symmetric) Rost correspondences.
Writing $\rho_{F(X)}$ as in Definition \ref{def rost corr}, we have
$(\rho^t\compose\rho)_{F(X)}=\chi_1\times[X_{F(X)}]+[X_{F(X)}]\times\chi_1$
(and $(\rho\compose\rho^t)_{F(X)}=\chi_2\times[X_{F(X)}]+[X_{F(X)}]\times\chi_2$).
\end{rem}

\begin{lemma}
Assume that the variety $X$ is projective homogeneous.
Let $\rho\in\Ch_d(X\times X)$ be a projector.
If there exists a field extension $E/F$ such that
$\rho_E=\chi_1\times[X_E]+[X_E]\times\chi_2$
for some $0$-cycle classes $\chi_1,\chi_2\in\Ch_0(X_E)$ of degree $1$,
then $\rho$ is a Rost projector.
\end{lemma}

\begin{proof}
According to \cite[\Small Theorem 7.5]{MR2110630},
there exist some integer $n\geq0$ and for $i=1,\dots,n$ some integers
$r_i>0$ and some projective homogeneous varieties $X_i$ satisfying
$\dim X_i+r_i<d$ such that for
$M=\Oplus_{i=1}^nM(X_i)(r_i)$
the motive $M(X)_{F(X)}$ decomposes
as $\F_2\oplus M\oplus\F_2(d)$.
Since there is no non-zero morphism between different summands of this three terms decomposition,
the ring $\End M(X)$ decomposes in the product of rings
$$
\End\F_2\times\End M\times\End\F_2(d)=
\F_2\times\End M\times\F_2.
$$

Let $\chi\in\Ch_0(X_{F(X)})$ be a $0$-cycle class of degree $1$.
We set
\begin{multline*}
\rho'=\chi\times[X_{F(X)}]+[X_{F(X)}]\times\chi\in
\F_2\times\F_2
\\
\subset
\F_2\times\End M\times\F_2=\End M(X)_{F(X)}=
\Ch_d(X_{F(X)}\times X_{F(X)})
\end{multline*}
and
we show that $\rho_{F(X)}=\rho'$.
The difference $\varepsilon=\rho_{F(X)}-\rho'$ vanishes over $E(X)$.
Therefore $\varepsilon$ is a nilpotent element of $\End M$.
Choosing a positive integer $m$ with $\varepsilon^m=0$,
we get
\begin{equation*}
\rho_{F(X)}=\rho_{F(X)}^m=(\rho'+\varepsilon)^m=
(\rho')^m+\varepsilon^m=(\rho')^m=\rho'.
\qedhere
\end{equation*}
\end{proof}

\begin{lemma}
Let $\rho\in\Ch_d(X\times X)$ be a projector.
The motive $(X,\rho)$ is isomorphic to $\F_2\oplus\F_2(d)$ iff
$\rho=\chi_1\times[X]+[X]\times\chi_2$ for some
some $0$-cycle classes $\chi_1,\chi_2\in\Ch_0(X)$  of degree $1$.
\end{lemma}

\begin{proof}
A morphism $\F_2\oplus\F_2(d)\to(X,\rho)$ is given by
some
$$
f\in\Hom\big(\F_2,M(X)\big)=\Ch_0(X)\;
\text{ and }\;
f'\in\Hom\big(\F_2(d),M(X)\big)=\Ch_d(X).
$$
A morphism in the inverse direction is given by some
$$
g\in\Hom(M(X),\F_2)=\Ch^0(X)\;
\text{ and }\;
g'\in\Hom(M(X),\F_2(d))=\Ch^d(X).
$$
The two morphisms
$\F_2\oplus\F_2(d)\leftrightarrow(X,\rho)$
are mutually inverse isomorphisms iff
$\rho=f\times g+f'\times g'$ and $\deg_{X/F}(fg)=1=\deg_{X/F}(f'g')$.
The degree condition means that $f'=[X]=g$ and $\deg_{X/F}(f)=1=\deg_{X/F}(g')$.
\end{proof}

\begin{cor}
\label{vid}
If $X$ is projective homogeneous and $\rho$ is a projector on $X$ such that
$$(X,\rho)_E\simeq\F_2\oplus\F_2(d)$$ for some field extension $E/F$, then
$\rho$ is a Rost projector.
\qed
\end{cor}

A smooth complete variety is called {\em anisotropic}, if the degree of its
any closed point is even.

\begin{lemma}[{\cite[\Small Lemma 9.2]{markus}, cf. \cite[\Small proof of Lemma 6.2]{semenov}}]
\label{sol'}
Assume that $X$ is anisotropic and possesses a Rost correspondence $\rho$.
Then for any integer $i\ne d$ and any elements
$\alpha\in\Ch_i(X)$ and $\beta\in\Ch^i(X_{F(X)})$, the image of the product
$\alpha_{F(X)}\cdot\beta\in\Ch_0(X_{F(X)})$ under the degree
homomorphism $\deg_{X_{F(X)}/F(X)}:\Ch_0(X_{F(X)})\to\F_2$ is $0$.
\end{lemma}

\begin{proof}
Let $\gamma\in\Ch^i(X\times X)$ be a preimage of $\beta$ under the surjection
$$
\xi_X^*:\Ch^i(X\times X)\to\Ch^i(\Spec F(X)\times X)
$$
(where $\xi^*_X$ is as defined in the proof of Proposition \ref{prop}).
We consider the $0$-cycle class
$$
\delta=\rho\cdot([X]\times\alpha)\cdot\gamma\in\Ch_0(X\times X).
$$
Since $X$ is anisotropic, so is $X\times X$, and it follows that
$\deg_{(X\times X)/F}\delta=0$.
Therefore it suffices to show that
$\deg_{(X\times X)/F}\delta=\deg_{X_{F(X)}/F(X)}(\alpha_{F(X)}\cdot\beta)$.

We have $\deg_{(X\times X)/F}\delta=\deg_{(X\times X)_{F(X)}/F(X)}(\delta_{F(X)})$ and
\begin{multline*}
\delta_{F(X)}=
(\chi_1\times[X_{F(X)}]+[X_{F(X)}]\times\chi_2)\cdot([X_{F(X)}]\times\alpha_{F(X)})\cdot
\gamma_{F(X)}=\\
(\chi_1\times[X_{F(X)}])\cdot([X_{F(X)}]\times\alpha_{F(X)})\cdot\gamma_{F(X)}
\end{multline*}
(because $i\ne d$)
where $\chi_1,\chi_2\in\Ch_0(X_{F(X)})$ are as in Definition \ref{def rost corr}.
For the first projection $\pr_1:X_{F(X)}\times X_{F(X)}\to X_{F(X)}$ we have
$$
\deg_{(X\times X)_{F(X)}/F(X)}\delta_{F(X)}=
\deg_{X_{F(X)}/F(X)}(\pr_1)_*(\delta_{F(X)})
$$
and by the projection formula
$$
(\pr_1)_*(\delta_{F(X)})=
\chi_1\cdot(\pr_1)_*\big(([X_{F(X)}]\times\alpha_{F(X)})\cdot\gamma_{F(X)}\big).
$$
Finally,
$$
(\pr_1)_*\big(([X_{F(X)}]\times\alpha_{F(X)})\cdot\gamma_{F(X)}\big)=
\mult\big(([X_{F(X)}]\times\alpha_{F(X)})\cdot\gamma_{F(X)}\big)\cdot[X_{F(X)}]
$$
and
$$
\mult\big(([X_{F(X)}]\times\alpha_{F(X)})\cdot\gamma_{F(X)}\big)=
\mult\big(([X]\times\alpha)\cdot\gamma\big).
$$
Since
$\mult\chi=\deg_{X_{F(X)}/F(X)}\xi_X^*(\chi)$
for any element $\chi\in\Ch_d(X\times X)$ by \cite[\Small Lemma 75.1]{EKM},
it follows that
\begin{equation*}
\mult\big(([X]\times\alpha)\cdot\gamma\big)=\deg(\alpha_{F(X)}
\cdot\beta).\qedhere
\end{equation*}
\end{proof}

For anisotropic $X$,
we consider the homomorphism
$\deg\!/2:\Ch_0(X)\to\F_2$ induced by the homomorphism
$\CH_0(X)\to\Z$, $\alpha\mapsto\deg(\alpha)/2$.

\begin{cor}
\label{cor-rost1}
Assume that $X$ is anisotropic and possesses a Rost correspondence.
Then for any integer $i\ne d$ and any elements
$\alpha\in\Ch_i(X)$ and $\beta\in\Ch^i(X)$ with $\beta_{F(X)}=0$ one has
$(\deg\!/2)(\alpha\cdot\beta)=0$.
\end{cor}

\begin{proof}
Let $\beta'\in\CH^i(X)$ be an integral representative of $\beta$.
Since $\beta_{F(X)}=0$, we have $\beta'_{F(X)}=2\beta''$ for some
$\beta''\in\CH^i(X_{F(X)})$.
Therefore
$$
(\deg\!/2)(\alpha\cdot\beta)=\deg_{X_{F(X)}/F(X)}\big(\alpha_{F(X)}
\cdot(\beta''\mod2)\big)=0
$$
by Lemma \ref{sol'}.
\end{proof}

\begin{cor}
\label{addition}
Assume that $X$ is anisotropic and possesses a Rost correspondence $\rho$.
For any integer $i\not\in\{0,d\}$ and any $\alpha\in\Ch_i(X)$ and $\beta\in\Ch^i(X)$
one has $$(\deg\!/2)\big((\alpha\times\beta)\cdot\rho\big)=0.$$
\end{cor}

\begin{proof}
Let $\alpha'\in\CH_i(X)$ and $\beta'\in\CH^i(X)$ be integral representatives of $\alpha$ and $\beta$.
Let $\rho'\in\CH_d(X\times X)$ be an integral representative of $\rho$.
It suffices to show that  the degree of the $0$-cycle class $(\alpha'\times\beta')\cdot\rho'\in\CH_0(X\times X)$ is divisible by $4$.

Let $\chi_1$ and $\chi_2$ be as in Definition \ref{def rost corr}.
Let $\chi'_1,\chi'_2\in\CH_0(X_{F(X)})$ be integral representatives of $\chi_1$
and $\chi_2$.
Then $\rho'_{F(X)}=\chi'_1\times [X_{F(X)}]+[X_{F(X)}]\times\chi'_2 +2\gamma$
for some $\gamma\in\CH_d(X_{F(X)}\times X_{F(X)})$.
Therefore (since $i\not\in\{0,d\}$)
$$
(\alpha'_{F(X)}\times\beta'_{F(X)})\cdot\rho'_{F(X)}=
2(\alpha'_{F(X)}\times\beta'_{F(X)})\cdot\gamma.
$$
Applying the projection $\pr_1$ onto the first factor and the projection formula,
we get twice the element
$\alpha'_{F(X)}\cdot(\pr_1)_*\big(([X_{F(X)}]\times\beta'_{F(X)})\cdot\gamma\big)$
whose degree is even by Lemma \ref{sol'}
(here we use once again the condition that $i\ne d$).
\end{proof}

\begin{lemma}
\label{rho kvadrat}
Assume that $X$ is anisotropic and possesses a Rost correspondence $\rho$.
Then $(\deg\!/2)(\rho^2)=1$.
\end{lemma}

\begin{proof}
Let $\chi_1$ and $\chi_2$ be as in Definition \ref{def rost corr}.
Let $\chi'_1,\chi'_2\in\CH_0(X_E)$ be integral representatives of $\chi_1$
and $\chi_2$.
The degrees of $\chi'_1$ and $\chi'_2$ are odd.
Therefore, the degree of the cycle class
$$
(\chi'_1\times[X_{F(X)}]+[X_{F(X)}]\times\chi'_2)^2=
2(\chi'_1\times\chi'_2)\in\CH_0(X_{F(X)}\times X_{F(X)})
$$
is not divisible by $4$.

Let $\rho'\in\CH_d(X\times X)$ be an integral representative of $\rho$.
Since $\rho'_{F(X)}$ is $\chi'_1\times [X_{F(X)}]+[X_{F(X)}]\times\chi'_2$ modulo $2$,
$(\rho'_{F(X)})^2$ is $(\chi'_1\times [X_{F(X)}]+[X_{F(X)}]\times\chi'_2)^2$ modulo $4$.
Therefore $(\deg\!/2)(\rho^2)=1$.
\end{proof}


\begin{thm}
[{\cite[\Small Theorem 9.1]{markus}, see also \cite[\Small proof of Lemma 6.2]{semenov}}]
\label{markus1}
Let $X$ be an anisotropic smooth complete geometrically
irreducible variety of a positive dimension $d$ over a field $F$ of characteristic $\ne2$ possessing a Rost correspondence.
Then the degree of the highest Chern class $c_d(-T_X)$, where $T_X$ is
the tangent bundle on $X$, is not divisible by $4$.
\end{thm}

\begin{proof}
In this proof, we write $c_\bullet(-T_X)$ for the total Chern class
$\in\Ch(X)$ in the Chow group with coefficient in $\F_2$.
It suffices to show that $(\deg\!/2)(c_d(-T_X))=1$.

Let $\Steen_\bullet^X:\Ch(X)\to\Ch(X)$ be the modulo $2$ homological Steenrod operation,
\cite[\Small \S59]{EKM}.
We have a commutative diagram
\begin{equation*}
\label{nk3:vot tak diagramka}
\begin{diagram}
\node{}\node{\Ch_d(X\times X)}\arrow{sw,t}{(\pr_1)_*}
\arrow[2]{s,r}{\Steen_d^{X\times X}}\node{}\\
\node{\Ch_d(X)}\arrow[2]{s,l}{\Steen_d^X}\node{}
\node{}
\\
\node{}\node{\Ch_0(X\times X)}\arrow{sw,t}{(\pr_1)_*}
\arrow[2]{s,r}{\deg\!/2}
\arrow{se,l}{(\pr_2)_*}
\node{}\\
\node{\Ch_0(X)}\arrow{se,b}{\deg\!/2}\node{}
\node{\Ch_0(X)}
\arrow{sw,r}{\deg\!/2}
\\
\node{}\node{\F_2}\node{}
\end{diagram}
\end{equation*}
Since $(\pr_1)_*(\rho)=[X]$ and $\Steen_d^X([X])=c_d(-T_X)$
\cite[\Small formula (60.1)]{EKM},
it suffices to show that
$$
(\deg\!/2)\big(\Steen^{X\times X}_d(\rho)\big)=1.
$$

We have $\Steen_\bullet^{X\times X}=c_\bullet(-T_{X\times X})\cdot\Steen^\bullet_{X\times X}$, where $\Steen^\bullet$ is the cohomological Steenrod operation, \cite[\Small \S61]{EKM}.
Therefore
$$
\Steen^{X\times X}_d(\rho)=\sum_{i=0}^dc_{d-i}(-T_{X\times X})\cdot\Steen^i_{X\times X}(\rho).
$$

The summand with $i=d$ is $\Steen^d_{X\times X}(\rho)=\rho^2$ by
\cite[\Small Theorem 61.13]{EKM}.
By Lemma \ref{rho kvadrat}, its image under $\deg\!/2$ is $1$.

Since $c_\bullet(-T_{X\times X})=c_\bullet(-T_X)\times c_\bullet(-T_X)$
and $\Steen^0=\id$,
the summand with $i=0$ is
$$
\left(\sum_{j=0}^d c_j(-T_X)\times c_{d-j}(-T_X)\right)\cdot\rho.
$$
Its image under $\deg\!/2$ is $0$ because
\begin{multline*}
(\deg\!/2)\Big(\big(c_0(-T_X)\times c_d(-T_X)\big)\cdot\rho\Big)
=(\deg\!/2)(c_d(-T_X))=
\\
(\deg\!/2)\Big(\big(c_d(-T_X)\times c_0(-T_X)\big)\cdot\rho\Big)
\end{multline*}
while
for $j\not\in\{0,d\}$, we have
$(\deg\!/2)\Big(\big(c_j(-T_X)\times c_{d-j}(-T_X)\big)\cdot\rho\Big)=0$
by Corollary \ref{addition}.

Finally, for any $i$ with $0<i<d$
the $i$th summand is the sum
$$
\sum_{j=0}^{d-i}\big(c_j(-T_X)\times c_{d-i-j}(-T_X)\big)\cdot\Steen^i_{X\times X}(\rho).
$$
We shall show that for any $j$ the image of the $j$th summand under $\deg\!/2$ is $0$.
Note that the image under $\deg\!/2$ coincides with the image under the composition
$(\deg\!/2)\compose(\pr_1)_*$ and also under the composition $(\deg\!/2)\compose(\pr_2)_*$
(look at the above commutative diagram).
By the projection formula we have
\begin{multline*}
(\pr_1)_*\Big(\big(c_j(-T_X)\times c_{d-i-j}(-T_X)\big)\cdot\Steen^i_{X\times X}(\rho)\Big)=\\c_j(-T_X)\cdot(\pr_1)_*\Big(\big([X]\times c_{d-i-j}(-T_X)\big)\cdot\Steen^i_{X\times X}(\rho)\Big)
\end{multline*}
and the image under $\deg\!/2$ is $0$ for positive $j$ by
Corollary \ref{cor-rost1} applied to $\alpha=c_j(-T_X)$ and
$\beta=(\pr_1)_*\Big(\big([X]\times c_{d-i-j}(-T_X)\big)\cdot\Steen^i_{X\times X}(\rho)\Big)$.
Corollary \ref{cor-rost1} can be indeed applied, because since $\rho_{F(X)}=\chi_1\times[X_{F(X)}]+[X_{F(X)}]\times\chi_2$ and $i>0$,
we have $\Steen^i_{(X\times X)_{F(X)}}(\rho)_{F(X)}=0$ and therefore
$\beta_{F(X)}=0$.

For $j=0$ we use the projection formula for $\pr_2$ and Corollary \ref{cor-rost1} with
$\alpha=c_{d-i}(-T_X)$ and $\beta=(\pr_2)_*\big(\Steen^i_{X\times X}(\rho)\big)$.
\end{proof}

\begin{rem}
\label{char2}
The reason of the characteristic exclusion in Theorem \ref{markus1} is that
its proof makes use of Steenrod operations on Chow groups with coefficients
in  $\F_2$ which are not available in characteristic $2$.
\end{rem}

We would like to mention

\begin{lemma}[{\cite[\Small Lemma 9.10]{markus}}]
\label{markus2}
Let $X$ be an anisotropic smooth complete equidimensional variety over a field of arbitrary characteristic.
If $\dim X+1$ is not a power of $2$, then the degree of the integral $0$-cycle class $c_{\dim X}(-T_X)\in\CH_0(X)$ is divisible by $4$.
\end{lemma}

\section
{Motivic decompositions of some isotropic varieties}
\label{Motivic decompositions of some isotropic varieties}

The coefficient ring $\Lambda$ is
$\F_2$ in this section.
Throughout this section, $D$ is a central division
$F$-algebra of degree $2^r$ with some positive integer $r$.

We say that motives $M$ and $N$ are {\em quasi-isomorphic} and write
$M\approx N$, if there exist
decompositions $M\simeq M_1\oplus\dots\oplus M_m$ and
$N\simeq N_1\oplus\dots\oplus N_n$ such that
$$
M_1(i_1)\oplus\dots\oplus M_m(i_m)\simeq
N_1(j_1)\oplus\dots\oplus N_n(j_n)
$$
for some (shift) integers $i_1,\dots,i_m$
and $j_1,\dots,j_n$.

We shall use the following

\begin{thm}[{\cite[\Small Theorems 3.8 and 4.1]{outmot}}]
\label{cd(gSB)}
For any  integer $l=0,1,\dots,r$,
the upper indecomposable summand $M_l$ of the motive of the generalized
Severi-Brauer variety $X(2^l;D)$ is lower.
Besides of this, the motive of
any finite direct product of any generalized Severi-Brauer varieties of $D$
is quasi-isomorphic to a finite sum of $M_l$ (with various
$l$).
\end{thm}

For the rest of this section, we fix an orthogonal involution
on the algebra $D$.

\begin{lemma}
\label{stand}
Let $n$ be a positive integer.
Let $h$ be a hyperbolic hermitian form on the right $D$-module $D^{2n}$
and let $Y$ be the variety
$\SB(n\deg D;(D^{2n},h))$ (of the {\em maximal} totally isotropic submodules).
Then the motive $M(Y)$ is isomorphic to a
finite sum of several shifted copies of
the motives $M_0,M_1,\dots,M_r=\Lambda$ including one non-shifted copy of $\Lambda$.
\end{lemma}

\begin{proof}
By
\cite[{\cyr Sledstvie} 15.9]{MR1758562}
(cf. \cite{MR2110630} or \cite{MR2178658}),
the motive of the variety $Y$ is quasi-isomorphic to the motive of the ``total''
variety
$$
\SB(*;D^n)=\coprod_{i\in\Z}\SB(i;D^n)=\coprod_{i=0}^{4n}\SB(i;D^n)
$$
of $D$-submodules in $D^n$.
Furthermore, $M(\SB(*;D^n))\approx M(\SB(*;D))^{\otimes n}$
by \cite[{\cyr Sledstvie} 10.10]{MR1758562} (cf. \cite{MR2110630} or \cite{MR2178658}).
We finish by Theorem \ref{cd(gSB)}.
The non-shifted copy of $\Lambda$ is obtained as
$\Lambda=M(\SB(0;D^n))=M(\SB(0;D))^{\otimes n}$.
\end{proof}

As before,
we write $\Ch(-)$ for the Chow group $\CH(-;\F_2)$ with coefficients in $\F_2$.
We recall that a smooth complete variety is called {\em anisotropic}, if the degree of its
any closed point is even (the empty variety is anisotropic).
The following statement is a particular case of \cite[\Small Lemma 2.21]{outmot}.

\begin{lemma}
\label{anisotropic}
Let $Z$ be an anisotropic $F$-variety with a projector $p\in\Ch_{\dim Z}(Z\times Z)$ such that
the motive $(Z,p)_L\in\CM(L,\F_2)$ for a field extension $L/F$
is isomorphic to a finite sum of Tate motives.
Then the number of the Tate summands is even.
In particular, the motive in $\CM(F,\F_2)$
of any anisotropic $F$-variety does not contain a Tate summand.
\end{lemma}

\begin{proof}
Mutually inverse isomorphisms between $(Z,p)_L$ and a sum of, say, $n$ Tate summands,
are given by two sequences of homogeneous elements
$a_1,\dots,a_n$ and $b_1,\dots,b_n$ in $\Ch(Z_L)$ with
$p_L=a_1\times b_1+\dots+a_n\times b_n$ and such that for any $i,j=1,\dots,n$
the degree $\deg(a_ib_j)$ is $0$ for $i\ne j$ and $1\in\F_2$ for $i=j$.
The pull-back of $p$ via the diagonal morphism of $Z$ is therefore a $0$-cycle class
on $Z$ of degree $n$ (modulo $2$).
\end{proof}

\begin{lemma}
\label{Y'}
Let $n$ be a positive integer.
Let $h'$ be a hermitian form on the right $D$-module $D^n$
such that $h'_L$ is anisotropic for any finite odd degree field extension
$L/F$.
Let $h$ be the hermitian form on the right $D$-module $D^{n+2}$ which is the
orthogonal sum of $h'$ and a hyperbolic $D$-plane.
Let $Y'$ be the variety of totally isotropic submodules of $D^{n+2}$
of reduced dimension $2^r$ ($=\ind D$).
Then the complete motivic decomposition of $M(Y')\in\CM(F,\F_2)$
(cf. Corollary \ref{krull-schmidt})
contains one summand $\F_2$,
one summand $\F_2(\dim Y')$, and does not contain any other Tate motive.
\end{lemma}

\begin{proof}
According to \cite[{\cyr Sledstvie} 15.9]{MR1758562},
$M(Y')$ is quasi-isomorphic to the sum of the motives of
the products
$$
\SB(i\subset j;D)\times \SB(j-i;(D^n,h'))
$$
where $i,j$ run
over all integers (the product is non-empty iff $0\leq i\leq j\leq 2^r$).
The choices $i=j=0$ and $i=j=2^r$ give the
summands $\F_2$ and $\F_2(\dim Y')$.
The variety obtained by any other choice of $i,j$ is anisotropic
(the variety with $i=0,j=2^r$ is anisotropic by the assumption involving the odd degree
field extensions),
and we are done by Lemma \ref{anisotropic}.
\end{proof}

\section
{Proof of Main theorem}
\label
{Proof of Theorem}

We fix a central simple algebra $A$ of index $>1$ with a non-hyperbolic orthogonal
involution $\sigma$.
Since the involution is an isomorphism of $A$ with its dual,
the exponent of $A$ is $2$; therefore, the index of $A$ is a
power of $2$, say, $\ind A=2^r$ for a positive integer $r$.
We assume that $\sigma$ becomes hyperbolic over the function field of the
Severi-Brauer variety of $A$ and we are looking for a contradiction.


According to \cite[\Small Theorem 3.3]{isotropy}, $\coind A=2n$ for some integer $n\geq1$.
We assume that Main theorem (Theorem \ref{main}) is already proven for all algebras (over
all fields) of index $<2^r$ as well as for all algebras of index $2^r$ and coindex $<2n$.

Let $D$ be a central division algebra Brauer-equivalent to $A$.
Let $X_0$ be the Severi-Brauer variety of $D$.
Let us fix an orthogonal involution on $D$ and an isomorphism of $F$-algebras
$A\simeq\End_D(D^{2n})$.
Let $h$ be
a hermitian form on the right $D$-module $D^{2n}$ such that
$\sigma$ is adjoint to $h$.
Then $h_{F(X_0)}$ is hyperbolic.
Since the anisotropic kernel of $h$ also becomes hyperbolic over $F(X_0)$,
our induction hypothesis ensures that $h$ is anisotropic.
Moreover, $h_L$ is hyperbolic for any field extension $L/F$ such that $h_L$ is
isotropic.
It follows by \cite[\Small Proposition 1.2]{MR1055648}
that $h_L$ is anisotropic for any finite odd degree field
extension $L/F$.

Let $Y$ be the variety of totally isotropic submodules in $D^{2n}$ of reduced
dimension $n\deg D$.
(The variety $Y$ is a twisted form of the variety of maximal totally
isotropic subspaces of a quadratic form studied in \cite[\Small Chapter XVI]{EKM}.)
It is isomorphic to the variety of totally isotropic right ideals in $A$ of
reduced dimension $(\deg A)/2$ (=$n2^r$).
Since $\sigma$ is hyperbolic over $F(X_0)$ and the field $F$ is algebraically closed
in $F(X_0)$ (because the variety $X_0$ is geometrically integral), the
discriminant of $\sigma$ is trivial.
Therefore the variety $Y$ has two connected components
$Y_+$ and $Y_-$ corresponding to the components $C_+$ and $C_-$
(cf. \cite[\Small Theorem 8.10]{EKM})
of the Clifford algebra $C(A,\sigma)$.

Since $\sigma_{F(X_0)}$ is hyperbolic, $Y(F(X_0))\ne\emptyset$.
Since the varieties $Y_+$ and $Y_-$ become isomorphic over $F(X_0)$,
each of them has an $F(X_0)$-point.

The central simple algebras $C_+$ and $C_-$ are related with $A$ by the formula
\cite[\Small (9.14)]{MR1632779}:
$$
[C_+]+[C_-]=[A]\in\Br(F).
$$
Since $[C_+]_{F(X_0)}=[C_-]_{F(X_0)}=0\in\Br(F(X_0))$, we have
$[C_+],[C_-]\in\{0,[A]\}$ and it follows that $[C_+]=[A]$, $[C_-]=0$ up
to exchange of the indices $+,-$.

By the index reduction formula for the varieties $Y_+$ and $Y_-$ of
\cite[\Small page 594]{MR1415325}, we have (up
to exchange of the indices $+,-$):
$\ind D_{F(Y_+)}=\ind D$, $\ind D_{F(Y_-)}=1$.
We replace $Y$ by the component of $Y$ whose function field does not reduce the
index of $D$.
Note that the variety $Y$ is projective homogeneous.

The coefficient ring $\Lambda$ is $\F_2$ in this section.
We use the $F$-motives $M_0,\dots,M_r$ introduced in Theorem
\ref{cd(gSB)}.
Note that for any field extension $E/F$ such that $D_E$ is still a division algebra,
we also have the $E$-motives $M_0,\dots,M_r$.

\begin{lemma}
\label{R2}
The motive of $Y$ decomposes as $R_1\oplus R_2$, where $R_1$ is
quasi-isomorphic to a finite sum
of several copies of the motives $M_0,\dots,M_{r-1}$, and where
$(R_2)_{F(Y)}$ is isomorphic to a finite sum of Tate motives including
one exemplar of $\F_2$.
\end{lemma}

\begin{proof}
According to Lemma \ref{stand} and since $D_{F(Y)}$ is a division algebra,
the motive $M(Y)_{F(Y)}$ is isomorphic to a
sum of several shifted copies of the $F(Y)$-motives $M_0,\dots,M_r=\F_2$
including
a copy of $\F_2$.
If for some $l=0,\dots,r-1$ there is at least one copy of $M_l$ (with a shift $j\in\Z$) in the
decomposition,
let us apply Proposition \ref{prop} taking as $X$ the variety $X_l=X(2^l,D)$,
taking as $M$ the motive $M(Y)(-j)$, and taking as $E$ the function field
$F(Y)$.

Since, as mentioned already, $D_E$ is a division algebra, condition (2) of Proposition \ref{prop} is fulfilled.
Since $\ind D_{F(X)}<2^r$, the hermitian form $h_{F(X)}$ is hyperbolic by
the induction hypothesis;
therefore the variety $Y_{F(X)}$ is rational and condition (1) of Proposition \ref{prop} is
fulfilled as well.

It follows that the $F$-motive $M_l$ is a summand of $M(Y)(-j)$.
Let now $M$ be the complement summand of $M(Y)(-j)$.
By Corollary \ref{krull-schmidt}, the complete decomposition of $M_{F(Y)}$ is
the complete decomposition of $M(Y)(-j)_{F(Y)}$ with one copy of $M_l$ erased.
If $M_{F(Y)}$ contains
one more copy of a shift of $M_l$ (for some $l=0,\dots,r-1$),
we once again apply Proposition
\ref{prop} to the variety $X_l$ and an appropriate shift of
$M$.
Doing this until we can, we get the desired decomposition in the end.
\end{proof}

Now let us consider a minimal right $D$-submodule $V\subset D^{2n}$ such that
$V$ becomes isotropic over a finite odd degree field extension of $F(Y)$.
We set $v=\dim_DV$.
Clearly, $2\leq v\leq n+1$. \footnote{One probably always has $v=n+1$ here,
but we do not need to know the precise value of $v$.}
For $n>1$, let $Y'$ be the variety of totally isotropic submodules in $V$ of reduced
dimension $2^r$  (that is, of ``$D$-dimension'' $1$).
For $n=1$ we set $Y'=Y$.
Note that the variety $Y'$ is projective homogeneous.

The variety $Y'$ is irreducible and has an even positive dimension.
Moreover, the variety $Y'$ is anisotropic (because the hermitian form $h$ is anisotropic
and remains anisotropic over any finite odd degree field extension of the base
field).
Surprisingly, we can however prove the following

\begin{lemma}
\label{there is}
There is a Rost projector (Definition \ref{def rost corr}) on $Y'$.
\end{lemma}

\begin{proof}
By the construction of $Y'$, there exists a correspondence
of odd multiplicity (that is, of multiplicity $1\in\F_2$)
$\alpha\in\Ch_{\dim Y}(Y\times Y')$.
On the other hand, since $h_{F(Y')}$ is isotropic, $h_{F(Y')}$ is
hyperbolic and therefore there exist a rational map $Y'\RatM Y$ and a multiplicity
$1$ correspondence $\beta\in\Ch_{\dim Y'}(Y'\times Y)$ (e.g., the class of the closure of the graph of the rational map).
Since the summand $R_2$ of $M(Y)$ given by Lemma \ref{R2} is upper
(cf. Definition \ref{def-outer} and Lemma \ref{tate outer}),
by Lemma \ref{tuda-suda}
there is an upper summand of $M(Y')$ isomorphic to a summand of $R_2$.
%

Let $\rho\in\Ch_{\dim Y'}(Y'\times Y')$ be the projector giving this summand.
We claim that $\rho$ is a Rost projector.
We prove the claim by showing
that the motive $(Y',\rho)_{\tilde{F}}$ is isomorphic
to $\F_2\oplus\F_2(\dim Y')$, cf. Corollary \ref{vid},
where $\tilde{F}/F(Y)$ is a finite odd degree field extension such that
$V$ becomes isotropic over $\tilde{F}$.

Since $(R_2)_{F(Y)}$ is a finite sum of
Tate motives, the motive $(Y',\rho)_{\tilde{F}}$
is also a finite sum of Tate motives.
Since $(Y',\rho)_{\tilde{F}}$ is upper, the Tate motive $\F_2$ is included
(Lemma \ref{tate outer}).
Now, by the minimal choice of $V$, the hermitian form $(h|_V)_{\tilde{F}}$
satisfies the condition on $h$ in Lemma \ref{Y'}:
$(h|_V)_{\tilde{F}}$ is an orthogonal sum of a hyperbolic $D_{\tilde{F}}$-plane
and a hermitian form $h'$ such that $h'_L$ is anisotropic for any finite odd
degree field extension $L/\tilde{F}$ of the base field $\tilde{F}$ (otherwise any $D$-hyperplane $V'\subset V$
would become isotropic over some odd degree extension of $F(Y)$).
Therefore the complete motivic decomposition of $Y'_{\tilde{F}}$
has one copy of $\F_2$, one copy of $\F_2(\dim Y')$, and no other Tate summands.
By Corollary \ref{krull-schmidt} and anisotropy of the variety $Y'$ (see Lemma \ref{anisotropic}),
it follows that
\begin{equation*}
(Y',\rho)_{\tilde{F}}\simeq\F_2\oplus\F_2(\dim Y').
\qedhere
\end{equation*}
\end{proof}


Lemma \ref{there is} contradicts to the general results of \S\ref{Rost correspondences} (namely, to Theorem \ref{markus1} and Lemma \ref{markus2})
thus proving Main theorem (Theorem \ref{main}).
We can avoid the use of Lemma \ref{markus2} by
showing that $\deg c_{\dim Y'}(-T_{Y'})$ is divisible by $2^{2^r}$ for our variety $Y'$.
Indeed, let $K$ be the field $F(t_1,\dots,t_{v2^r})$ of rational functions over $F$
in $v2^r$ variables.
Let us consider the (generic) diagonal quadratic form $\<t_1,\dots,t_{v2^r}\>$ on the $K$-vector space
$K^{v2^r}$.
Let $Y''$ be the variety of $2^r$-dimensional totally isotropic subspaces
in $K^{v2^r}$.
The degree of any closed point on $Y''$ is divisible by $2^{2^r}$.
In particular, the integer $\deg c_{\dim Y''}(-T_{Y''})$ is divisible by
$2^{2^r}$.
Since over an algebraic closure $\bar{K}$ of $K$ the varieties $Y'$ and
$Y''$ become isomorphic,  we have
$$
\deg c_{\dim Y'}(-T_{Y'})=\deg c_{\dim Y''}(-T_{Y''}).
$$


\def\cprime{$'$}

\end{document}